\documentclass[12pt]{amsart}
\usepackage[utf8]{inputenc}
\usepackage{amssymb}
\usepackage{hyperref}
\theoremstyle{plain}
\newtheorem{thm}{Theorem}
\newtheorem{cor}{Corollary}
\newtheorem{lem}{Lemma}
\newtheorem{prop}{Proposition}

\theoremstyle{definition}
\newtheorem{defn}{Definition}

\newtheorem{rem}[thm]{Remark}
\newtheorem{ques}[thm]{Question}

\title{Decidable Fragments of Theories\\ in \\Field Arithmetic }

\date{\today\\
AMS 2010 Subject Classification: Primary 03B25. Secondary: 03C60, 11U05, 12L05}
\keywords{Decidable Fragments;Hilbertian field; PAC field}

\author{Chun-YU Lin}

\email{maxcylin@iis.sinica.edu.tw}

\address{Institute of Information Science, Academia Sinica, Taiwan.}
\thanks{This paper is the author's master thesis. He would like to express his gratitude for Professor Shih-Ping Tung's guidance.}
\begin{document}

\begin{abstract}
 In this paper, we show that the $\exists^1 \forall^1$ theories of Hilbertian fields with charateristic 0 and perfect Hilbertian fields are both decidable. We also prove that the $\forall^1 \exists^1$ theories of Hilbertian fields with charateristic 0, Hilbertian fields, PAC fields with characteristic 0, and PAC fields are all decidable.
\end{abstract}
\maketitle

\section{Introduction}
In mathematics, we often face problems take the form: find an effective procedure by means of which it can be determined in finite steps for each element of our interested set, whether or not the element satisfied the defining property. The solution of such problem usually consists of exhibiting  algorithmic-like arguments or proofs to demonstrate that procedure. The problems of this kind are called decision problems or decidability of theories which depends on the set we considered. 

Most of the decision problems in mathematics are solvable\cite{MOR}. However, there are still some decision problems that are unsolvable. For example, the word problem is unsolvable~\cite{BCL}. Hilbert's tenth problem over $\mathbb{Z}$ and $\mathbb{N}$ are both unsolvable~\cite{DMR}. For the first order theories, the cases are different. In 1931, K. G$\ddot{o}$del announced his famous incompleteness theorem which implies that the elementary theory of $\langle \mathbb{N},+,\cdot,0,1 \rangle$ and $\langle \mathbb{Z},+.\cdot,0,1 \rangle$ are undecidable~\cite{KG}. On the other hand, C.~H.~Langford proved in 1927 that the elementary theory of  $\langle \mathbb{N},\leq \rangle$ is decidable~\cite{CHL}. There are many elementary theories of various mathematical structure have been proved to be decidable or undecidable since then. We list some theories of fields which have been proved to be undecidable and refer to \cite{ELT} and \cite[Chapter 13 and 16]{JM} for exhaustive lists of decidable and undecidable theories.
\begin{thm}
\label{mthm4}
The following elementary theories in the language $\mathcal{L_{ring}}$ of ring are all undecidable.
\begin{enumerate}
\item The elementary theory of fields, \cite{JR},
\item The elementary theory of fields of characteristic 0 \cite{JR},
\item The elementary theory of algebraic number fields~\cite{JR1},
\item The elementary theory of Hilbertian fields~\cite{FJ}, 
\item The elementary theory of PAC fields~\cite{CDM}.
\end{enumerate}
\end{thm}
Even if we consider the same domain, different structures may have different decidability  of the elementary theories. The results proposed by G$\ddot{o}$del and Langford mentioned above are good examples. But for those structures whose elementary theories are undecidable, we may ask the following problem:
\begin{ques}
\label{que1}
What subsets of undecidable theories are decidable or undecidable ? We try to find the decidable fragments(i.e. dividing line) of decidability of different theories.
\end{ques}
For the question~\ref{que1}, there are two different approaches: different numbers of quantifier, different kinds of quantifier ( $\forall$ or $\exists$ ). To discuss these two approaches, we need some definition of terminology.

\begin{defn}
Let Q deote the quantifier $\forall  \mbox{ or } \exists $. We call $\varphi$ a $Q_1^m Q_2^n$ sentence for $m,n \in \mathbb{N}$ if and only if $\varphi$ is logically equivalent to a sentence of the form $Q_1x_1 \cdots Q_1x_m Q_2y_1\cdots Q_2y_n \varphi'(x_1, \ldots,x_m,y_1,\ldots,y_n) $ where $\varphi' $ is a quantifier-free formula. We call $\psi$ a  $Q_1^m Q_2^n$ equation if and only if $\psi$ is of the form $$Q_1x_1 \cdots Q_1x_m Q_2y_1\cdots Q_2y_n f(x_1, \ldots,x_m,y_1,\ldots,y_n)=0 $$ where f is a polynomial.
\end{defn}
For example, Hilbert's Tenth Problem can be formulated in the form of question as: Decide whether or not the set of $\exists ^n$ equations for all $n \geq 0$ over $\mathbb{N}$ and $\mathbb{Z}$ are decidable. For sets of $Q_1^m Q_2^n$ sentences, we formulate following definition.
\begin{defn}
Let Q denote the quantifier $\forall  \mbox{ or } \exists $. We call a subset of elementary theory $Th(K)$ (resp. $Th(\mathbf{K})$) of an mathematical structure K (resp. a class of mathematical structure $\mathbf{K}$) a $Q_1^mQ_2^n$ theory if it consists of $Q_1^mQ_2^n$ sentences which is true in K ( resp. true in all mathematical structures in $\mathbf{K}$).
\end{defn}
Of course, we can extend the definition for $Q_1^m Q_2^n$ equation and theory to $Q_1^{m_1} \cdots Q_n{m_n}$ for $m_1,\ldots,m_n \in \mathbb{N}$ and $Q_1,\ldots,Q_n \in \{ \forall ,\exists  \}$ like arithmetical hierarchy in recursion theory. But the there are few results in three or more alternative quantifiers of $Q_1^{m_1} \cdots Q_n{m_n}$ theory over other algebraic structures than $\mathbb{N}$ and $\mathbb{Z}$. So we mainly consider $Q_1^m Q_2^n$ equation and theory. Notice that if we know that a $Q_1^m Q_2^n$ theory of some mathematical structures (or a class of mathematical structures) is decidable, then so is the set of $Q_1^m Q_2^n$ equations. But if a $Q_1^m Q_2^n$ theory of some mathematical structures (or a class of mathematical structures) is undecidable, it may still happens that the set of $Q_1^m Q_2^n$ equations is decidable. Since Hilbert's tenth problem over $\mathbb{N}$ and $\mathbb{Z}$ are unsolvable, we may ask what is the least n such that the set of $\exists^n$ equations for all $n \geq 0$ over $\mathbb{N}$ and $\mathbb{Z}$ are $\mathbf{undecidable}$ ? This is the first approach of our question~\ref{que1}.
\begin{thm}
\begin{enumerate}
\item The set of  $\exists^n$ equations over $\mathbb{N}$ is undecidable for all $n \geq  \bf9$~\cite{JJ1}.
\item The set of  $\exists^n$ equations over $\mathbb{Z}$ is undecidable for all $n \geq \bf11$~\cite{SZW}.
\end{enumerate}
\end{thm}
Note that the set of $\exists$ equations over $\mathbb{N}$ and $\mathbb{Z}$ are decidable as~\cite{MR} indicates. Since the decidability of $\exists^2$ equations over $\mathbb{N}$ and $\mathbb{Z}$ are still unknown~\cite{DMR}, we still have no answer for $2 \leq n \leq 8$ in the case of $\mathbb{N}$ and for $2 \leq n \leq 10$ in the case of $\mathbb{Z}$. For the second approach, we can separate it into global and local directions.
\begin{defn}
The global and local approach of different quantifier of decidable fragments are defined as following:
\begin{enumerate}
\item The global direction consider the $Q_1^m Q_2^n$ sentences with m or n to be all natural numbers.
\item The local direction consider the $Q_1^m Q_2^n$ sentences with fixed natural number m,n.
\end{enumerate}
\end{defn}
Note that the $\Pi^0_1 $, $\Sigma^0_1$, $\Pi^0_2$, and $\Sigma^0_2$ sentences in recursion theory are all special cases of above definition.
\begin{thm}
We have following results about decidability of the set of $Q_1^m Q_2^n$ equations and in global direction.
\begin{enumerate}
\item The set of $\exists ^n$ equations for all $n \in \mathbb{N}$ over $\mathbb{N}$ and $\mathbb{Z}$ are both undecidable~\cite{DMR}. (For the decidability $\exists ^n$ equations with all $n \in \mathbb{N}$ over other commutative rings, see~\cite{BP1}), 

\item The set of $\forall^n \exists$ equations for all $n\geq 0$ over $\mathbb{Z}$ is decidable but the set of $\forall^n \exists$ equations for all $n \geq 0 $ over $\mathbb{N}$ is undecidable. Also, the set of $\forall^n \exists^2$ equations for all $n\geq 0 $ over $\mathbb{Z}$ is undecidable~\cite{Tung1}.
\end{enumerate}
\end{thm}
If we consider the global direction in $Q_1^m Q_2^n$ theory, the results are slightly different.
\begin{thm}
For all $m,n \in \mathbb{N}$, we have following results:
\begin{enumerate}
\item The $\forall^m \exists^n$ theories of $\mathbb{N}$ and $\mathbb{Z}$ are undecidable, respectively~\cite{Tung1},
\item The $\forall^m \exists^n$ theory of $\mathbb{Q}$ is undecidable~\cite{K},
\item The $\forall^m \exists^n$ theory of a number field are undecidable~\cite{JP},
\item The $\forall^m$ theories of fields and integral domains are both decidable~\cite{HO},
\item The $\exists^m$ theory of PAC fields is decidable~\cite{FJ}.
\end{enumerate}
\end{thm}
Now, for the local direction, we often consider the case in $\forall^m \exists^n$ theories so that the decidability of $\forall^m \exists^n$ equations are determined automatically.

\begin{thm}
We have the following results about decidability of $Q_1^m Q_2^n$ theory in global direction.
\begin{enumerate}
\item The set of $\forall^1 \exists^1$ equations over $\mathbb{N}$ and $\mathbb{Z}$ are both decidable~\cite{JJ2},
\item The $\forall^1 \exists^1$ and $\exists^1 \forall^1$ theory of $\mathbb{N},\mathbb{Z},\mathbb{Q}$ and an algebraic number field K are decidable, respectively~\cite{Tung2,Tung3},
\item The $\forall^1 \exists^1$ theory of algebraic number fields, fields of characteristic 0, and fields are decidable, respectively~\cite{Tung3},
\item The $\forall^1 \exists^1$ theory of integral domains is decidable~\cite{Tung4}, 
\item The $\forall^1 \exists^1$ theory of algebraic integer rings is decidable~\cite{Tung5}.
\end{enumerate}
\end{thm}

\section{Decidability of $\exists^1 \forall^1$ theories}
In this section, we investigate the decidability of $\exists^1 \forall^1$ theory of class of Hilbertian fields of characteristic 0 and perfect Hilbertian fields. 
From Theorem~\ref{mthm4} in the Introduction, we know that the elementary theory of Hilbertian fields and PAC fields are both undecidable. Also, from~\cite[P. 304-305]{JR2}, we know that the elementary theory of any class of fields which contains $\mathbb{Q}$ is undecidable. Then the elementary theories of Hilbertian fields of characteristic 0 and perfect Hilbertian fields are both undecidable too. 

With undecidable theories above, we want to investigate what fragments of theory Hilbertian fields and PAC fields are decidable ? In the following paragraphs, we show that the elementary theory of Hilbertian PAC fields of characteristic 0 is decidable first. Using this result, we  prove that the $\exists^1 \forall^1$ theory of Hilbertian fields of characteristic 0 and perfect Hilbertian fields are both decidable. 
Let $Frob(K,G)$ be the class of perfect Frobenius fields M containing a field K with all finite quotient $Im(Gal(M))$ of absolute Galois groups Gal(M) being isomorphic to finite quotient $Im(G)$. The following important theorem from \cite[Theorem 30.6.2]{FJ} shows that the elementary theory of the class Frob(K,G) of  is decidable.
\begin{thm}
\label{thm8}
Let K be a presented field with elimination theory and G a superprojective group such that Im(G) is primitive recursive. Then there exists a primitive recursive decision procedure for the theory of Frob(K,G).
\end{thm}
We prove that the decidability of $\omega-$free PAC fields of characteristic 0 using this key fact.
\begin{thm}
\label{thm9}
The elementary theory of $\omega$-free PAC fields of characteristic 0 is decidable.
\end{thm}
\begin{proof}
Denote the elementary theory of $\omega$-free PAC 
fields of characteristic 0 as Th($\omega$ -PAC0). From \cite{FJ} We know $\mathbb{Q}$ is a presented field with splitting algorithm through and has elimination theory by Proposition. Also, the free pro-$\mathcal{C}$ group of rank $\aleph_0$ with formation $\mathcal{C}$ of all finite groups, $\hat{F}_{\omega}$, is superprojective. Applying Theorem 24.8.1 in \cite{FJ}, we have Im($\hat{F}_{\omega}$)=$\mathcal{C}$ which is primitive recursive. Thus, we conclude that the elementary theory of Frob($\mathbb{Q},\hat{F}_{\omega}$), denoted as Th(Frob($\mathbb{Q},\hat{F}_{\omega}$)), is decidable through Theorem~\ref{thm8}.

Now to show Th($\omega$ -PAC0) = Th(Frob($\mathbb{Q},\hat{F}_{\omega})$), it's sufficient to show that the class of $\omega$-free PAC fields of characteristic 0 is the same class as $Frob(\mathbb{Q},\hat{F}_{\omega})$. Let K be an $\omega$-free PAC fields of characteristic 0. By definition, every finite embedding problem for the absolute Galois group Gal(K) is solvable. From Theorem 24.8.1 in \cite{FJ}, Gal(K) is isomorphic to $\hat{F}_{\omega}$. Using Theorem 24.8.1 in \cite{FJ}, we have Im(Gal(K))=$\mathcal{C}$ = Im($\hat{F}_{\omega}$) where $\mathcal{C}$ is the family of all finite groups and Gal(K) has the embedding property. Since K is of characteristic 0, K contains $\mathbb{Q}$. Therefore, we have $K \in Frob(\mathbb{Q},\hat{F}_{\omega})$. Conversely, suppose that M is a perfect Frobenius field containing $\mathbb{Q}$ with Im(Gal(M)) = Im($\hat{F}_{\omega}$) By definition, we know that Gal(M) has the embedding property and M is a PAC field. Then Gal(K) is isomorphic to $\hat{F}_{\omega}$. Note that Im(Gal(M)) consists of all finite groups through the equality setting above. This tells us that every finite embedding problem for Gal(M) is solvable. Therefore, M is an $\omega$-free PAC field of characteristic 0 and we finish our proof.
\end{proof}
\begin{cor}
\label{cor4}
The elementary theory of Hilbertian PAC fields of characteristic 0 is decidable.
\end{cor}
\begin{proof}
 From Theorem 5.10.3 in \cite{J}, we know that an $\omega$-free PAC field is the same as a Hilbertian PAC field. Note that Th($\omega$ -PAC0) is decidable by Theorem~\ref{thm9}. Then we get the desired result.
\end{proof}

Now, we prove some techniques which will be used in the proof following theorems.

\begin{lem}
\label{cor5}
Let K be a Hilbertian field and $\varphi(\bar{x},y)$ be a formula in disjunctive normal form. If $\forall \bar{x} \exists y \varphi(\bar{x},y)$ is true in K, then there exists an i and polynomial $F(\bar{x})$ and $G(\bar{x})$ $\not \equiv 0$ over K such that in $K[\bar{x},y]$, $G(\bar{x})y-F(\bar{x})$ are irreducible common factor of each $f_{i,j}(\bar{x},y)$, $1 \leq j \leq m_i$, but not a factor of any $g_{i,k}(\bar{x},y)$, $1 \leq k \leq n_i$.
\end{lem}
\begin{proof}
From Theorem 3.2 in \cite{Tung6}, we can find $p(\bar{x}) \in $K($\bar{X}$) such that $\varphi(\bar{x},p(\bar{x}))$ is true in K($\bar{X}$) since $\bar{\varphi}(y) \equiv \varphi(\bar{x},y)$ over K($\bar{X}$). Then there exists an $1 \leq i \leq s$ so that $\bigwedge_{j=1}^{m_i}f_{i,j}(\bar{x},p(\bar{x}))=0 \wedge \bigwedge_{k=1}^{n_i}g_{i,k}(\bar{x},p(\bar{x})) \neq 0$ is true in K($\bar{X}$). Write $p(\bar{x})= \frac{F(\bar{x})}{G(\bar{x})}$ where F($\bar{x}$) and G($\bar{x}$) are relatively prime and F($\bar{x}$),G($\bar{x}$)$\in$ K[$\bar{x}$] with G($\bar{x}$)$\not \equiv 0$. By Factor Theorem, $y-\frac{F(\bar{x})}{G(\bar{x})}$ is an irreducible common factor of each $f_{i,j}(\bar{x},y)$, but not a factor of any $g_{i,k}(\bar{x},y)$ for $1 \leq k \leq n_i$. Then Gauss lemma shows that G($\bar{x}$)y-F($\bar{x}$) is an irreducible common factor of each $f_{i,j}(\bar{x},y)$ but not a factor of any $g_{i,k}(\bar{x},y)$ for $1 \leq k \leq n_i$. This finishes our proof.
\end{proof}
Through this lemma, we can prove the preservation theorem of $\forall^1 \exists^1$ sentences over Hilbertian field.
\begin{prop}
\label{prop10}
Let K be a Hilbertian field contained in a field F. If K is algebraically closed in F, then F satisfies all $\forall^1 \exists^1 $ sentences true in K.
\end{prop}
\begin{proof}
Suppose that there exists an $\forall^1 \exists^1$ sentence which is true in K but false in F. Reduce $\varphi(x,y)$ to disjunctive normal form. Thus, $\varphi(x,y)$ is logically equivalent to $\bigvee_{i=1}^s[\bigwedge_{j=1}^{m_i}f_{i,j}(x,y)=0\wedge \bigwedge_{k=1}^{n_i}g_{i,k}(x,y) \neq 0]$ where $f_{i,j}(x,y)$ and $g_{i,k}(x,y)$ are polynomials over $\mathbb{Z}$. By Corollary~\ref{cor5}, there exists an i and polynomials F(x),G(x)$\in$K[x], with G(x)$\not \equiv 0$, such that G(x)y-F(x) is an irreducible common factor of the polynomials $f_{i,j}(x,y)$ for $1 \leq j \leq m_i$, but not a factor of $g_{i,k}(x,y)$ for $1 \leq k \leq n_i$. Since F$\models \exists x \forall y  \neg \varphi(x,y)$, choose $x'\in $F so that F$\models \forall y \neg \varphi(x',y)$. Because K$\models \forall x \exists y \varphi(x,y)$, $x' \in $F$-$K by preservation theorem of universal sentence. Since K is algebraically closed in F, G($x'$)$\neq 0$. Then $f_{i,j}(x,\frac{F(x)}{G(x)})=0$ for every $1 \leq j \leq m_i$ shows that F$\models \bigwedge_j  f_{i,j}(x',\frac{F(x')}{G(x')})=0$. Note that G(x)y-F(x) is not a factor of $g_{i,k}(x,y)$ in K(X)[y] for any $1 \leq k \leq n_i$, so $g_{i,k}(x,\frac{F(x)}{G(x)})\neq 0$ in K(X). Since x' is transcendental over K, we have $g_{i,k}(x',\frac{F(x')}{G(x')}) \neq 0$. Thus F$\models \bigwedge_j f_{i,j}(x',\frac{F(x')}{G(x')})=0 \wedge \bigwedge_k g_{i,k}(x',\frac{F(x')}{G(x')}) \neq 0$ for some i. Therefore, F$\models \exists y \varphi(x',y)$ with $y=\frac{F(x)}{G(x)}$ which contradicts to the assumption. This finishes our proof.
\end{proof}
Using Proposition~\ref{prop10}, Corollary~\ref{cor4}, we can get following series of decidable results about $\forall^1 \exists^1$ and $\exists^1 \forall^1$ theories of Hilbertian fields and PAC fields.
\begin{thm}
\label{thm16}
The $\exists^1 \forall^1$ of Hilbertian field of characteristic 0 is the same as the $\exists^1 \forall^1$ theory of Hilbertian PAC fields of characteristic 0.
\end{thm}
\begin{proof}
Since every Hilbertian PAC field of characteristic 0 is a Hilbertian field of characteristic 0, the $\exists^1 \forall^1$ theory of Hilbertian fields of characteristic 0 is contained in the $\exists^1 \forall^1 $ theory of Hilbertian PAC fields of characteristic 0. We want to show that these two sets are in fact the same. Suppose that there is an $\exists^1 \forall^1$ sentence $\exists x \forall y \varphi(x,y)$ with $\varphi(x,y)$ quantifier-free which is true in every Hilbertian PAC fields of characteristic 0 but false in a Hilbertian field K of characteristic 0.From Proposition 13.4.6 in \cite{FJ}, we know that Every field K has a regular extension F which is PAC and Hilbertian.Denote this Hilbetian PAC field with characteristic 0 as F. Using Proposition~\ref{prop10}, we have $F \models \forall x  \exists y \neg  \varphi(x,y)$ which contradicts to the assumption.
\end{proof}

\begin{cor}
\label{cor7}
The $\exists^1 \forall^1$ theory of Hilbertian fields of characteristic 0 is decidable
\end{cor}
\begin{proof}
From Corollary~\ref{cor4}, the $\exists^1 \forall^1$ theory of Hilbertian PAC fields of characteristic 0 is decidable. Then the Theorem~\ref{thm16} tell us the result we want.
\end{proof}

Actually, we can extend the above result to perfect Hilbertian fields. We quote the following theorem\cite{FJ} about decidability of elementary theory of perfect Frobenius fields.
\begin{thm}
\label{thm14} 
Let $\mathcal{C}$ be a primitive recursive full family of finite groups. Then the theory of perfect Frobenius fields M such that each Gal(M) is a pro-$\mathcal{C}$ group is primitive recursive.
\end{thm}
From this theorem, we have the decidability of perfect $\omega$-free PAC fields.
\begin{thm}
\label{thm12}
The elementary theory of perfect $\omega$-free PAC fields is decidable.
\end{thm}
\begin{proof}
From definition of Frobenius field and Corollary 24.8.2 in \cite{FJ}, we know that K is a perfect $\omega$-free PAC fields if and only if K is a perfect Frobenius field where Gal(K) is isomorphic to $\hat{F}_{\omega}$. Using Theorem 31.1.4 in \cite{FJ}, we have the desired result.
\end{proof}
\begin{thm}
\label{thm2}
The $\exists^1 \forall^1$ theory of perfect Hilbertian fields is decidable.
\end{thm}
\begin{proof}
Note first that the perfect Hilbertian PAC field is the same as the perfect $\omega$-free PAC field. We claim that the  $\exists^1 \forall^1$ theory of perfect Hilbertian fields and the $\exists^1 \forall^1$ theory of perfect Hilbertian PAC fields are identical. Then the result follows from Theorem~\ref{thm12}.

Since every perfect Hilbertian PAC fields is a perfect Hilbertian field, the $\exists^1 \forall^1$ theory of perfect Hilbertian fields is contained in the $\exists^1 \forall^1 $ theory of perfect Hilbertian PAC fields. Suppose that there is an $\exists^1 \forall^1$ sentence $\exists x \forall y \varphi(x,y)$ with $\varphi(x,y)$ quantifier-free which is true in all perfect Hilbertian PAC fields of but false in a perfect Hilbertian field K. Then there exists a regular extension F which is Hilbertain and PAC. Since K is perfect, F is also perfect. But Proposition~\ref{prop10} tells us that $F \models \forall x \exists y \neg \varphi(x,y)$ which contradicts to our assumption, This finishes our proof. 
\end{proof}

\section{Decidable $\forall^1 \exists^1$ theories}
Using the Theorems proved in previous section, we can give the proof of the decidability of $\forall^1 \exists^1$ theory of Hilbertian fields of characteristic 0. First, we need the following lemma about preservation of $\exists^n \forall^1$ sentences over Hilbertian fields.
\begin{lem}
\label{lem3}
Let K be a Hilbertian field and F is a regular extension of K. Then F satisfies all the $\exists^n \forall^1$ sentences true in K for arbitrary n.
\end{lem}
\begin{proof}
The proof of this lemma is essentially the same as the proof of Proposition 2.4 in \cite{Tung4}.
\end{proof}
\begin{thm}
\label{thm17}
The $\forall^1 \exists^1$ theory of Hilbertian fields of characteristic 0 is decidable.
\end{thm}
\begin{proof}
From Corollary~\ref{cor4}, the $\forall^1 \exists^1$ theory of Hilbertian PAC fields of characteristic 0 is decidable. If we shows that the $\forall^1 \exists^1$ theory of Hilbertian fields of characteristic 0 and the $\forall^1 \exists^1$ theory of Hilbertian PAC fields of characteristic 0 are identical, then we get the desired result. 

Since every Hilbertian PAC fields of characteristic 0 is a Hilbertian field of characteristic 0, the $\forall^1 \exists^1$ theory of Hilbertian fields of characteristic 0 is contained in the $\forall^1 \exists^1 $ theory of Hilbertian PAC fields of characteristic 0. To show that these two sets are in fact the same, suppose that there exists an  $\forall^1 \exists^1$ sentence $\forall x \exists y \varphi(x,y)$ with $\varphi(x,y)$ quantifier-free which is true in every Hilbertian PAC fields of characteristic 0 but false in a Hilbertian field K of characteristic 0. From Proposition 13.4.6 in \cite{FJ}, we can find a Hilbertian PAC field F which is regular extension of K. Then Lemma~\ref{lem3} tell us that $F \models \exists x \forall y \neg \varphi(x,y)$ which contradicts to our assumption.
\end{proof}
\begin{rem}
The $\forall^1 \exists^1$ theory of Hilbertian fields of characteristic 0 is contained in the $\forall^1 \exists^1$ theory of number fields and containing the $\forall^1 \exists^1$ theory of fields of characteristic 0. From Theorem 3.1 in \cite{Tung4}, we know that the $\forall^1 \exists^1$ theory of number fields and  the $\forall^1 \exists^1$ theory of fields of characteristic 0 are equal. So the three $\forall^1 \exists^1$ theories above are all equal and  we get the decidability result through Theorem 3.3 in \cite{Tung4}.
\end{rem}
\begin{cor}
The $\forall^n \exists^1$ theory of Hilbertian fields of characteristic 0 is decidable for arbitrary positive integer n.
\end{cor}
\begin{proof}
From Corollary~\ref{cor4}, the $\forall^n \exists^1$ theory of Hilbertian PAC fields of characteristic 0 is decidable for arbitrary positive integer n. Since Lemma~\ref{lem3} shows the preservation of $\exists^n \forall^1$ sentences over Hilbertian fields, we can directly modify the proof in Theorem~\ref{thm17} to show that the $\forall^n \exists^1$ theory of Hilbertian fields of characteristic 0 and the $\forall^n \exists^1$ theory of Hilbertian PAC fields of characteristic 0 are identical for arbitrary positive integer n. This proves the result we want.
\end{proof}
Next, we prove the decidability of $\forall^1 \exists^1$ theory of Hilbertian fields. From Exercise in \cite{CK}, we can easily get the following preservation theorem of $\exists^n \forall^m$ sentences for all $m,n \in \mathbb{N}$.
\begin{prop}
\label{prop11}
Let $\mathcal{A},\mathcal{B}$ be  $\mathcal{L}$-structures. If $\mathcal{A}$ is existentially closed in $\mathcal{B}$. then every $\exists^n \forall^m$ sentence for all $m,n \in \mathbb{N}$ true in $\mathcal{A}$ is also true in $\mathcal{B}$.
\end{prop}
\begin{proof}
If there exist an $\exists^n \forall^m$ sentence which is true in $\mathcal{A}$ but false in $\mathcal{B}$ for some $m,n \in \mathbb{N}$, then we can find that the negation of $\exists^n \forall^m$ sentence (i.e. $\forall^n \exists^m$ sentence) is also true in $\mathcal{A}$ by the Exercise in \cite{CK} and leads to a contradiction.
\end{proof}
But from Proposition 1 in \cite{RP}, we know that the existentially closed property can be characterized by transcendental extension in field theoretic sense.
\begin{thm}\cite{RP}
\label{thm3}
Let K be an infinite field and F is an extension field of K. If F is purely transcendental extension of K, then K is existentially closed in F.
\end{thm}
Note that the Theorem above require the fields to be infinite. But from \cite{FJ}, we know that the Hilbertian fields and PAC fields are infinite. 

Therefore, we first prove that the decidability of $\forall \exists $ theory of infinite fields.
\begin{prop}
\label{prop13}
The $\forall^1 \exists^1$ theory of infinite fields is decidable
\end{prop}
\begin{proof}
Let $\forall x \exists y \varphi(x,y)$ be an $\forall^1 \exists^1$ sentence. Reduce $\varphi(x,y)$ to disjunctive normal form:  $\bigvee_{i=1}^s[\bigwedge_{j=1}^{m_i}f_{i,j}(x,y)=0\wedge \bigwedge_{k=1}^{n_i}g_{i,k}(x,y) \neq 0]$. Suppose that $\forall x \exists y \varphi(x,y)$ is true in all fields of characteristic 0. Then $\forall x \exists y \varphi(x,y)$ in true in the rational number. Notice that $\mathbb{Q}$ is a Hilbertian field by Hilbert's irreducibility theorem. From Corollary~\ref{cor5}, there exist an i and polynomials $F(x),G(x) \in \mathbb{Q}[x]$, with $G(x) \not \equiv 0$, such that $G(x)y-F(x)$ is an irreducible common factor of $f_{i,j}(x,y)$ in $\mathbb{Q}[x,y]$ for $1 \leq j \leq m_i $ but not a factor of $g_{i,k}(x,y)$ for $1 \leq k \leq n_i$. As in $\mathbb{Q}$, we have a splitting algorithm to factor $f_{i,j}(x,y)$ and $g_{i,k}(x,y)$ for every $i,j,k$ over $\mathbb{Q}$ and looking for the polynomial $G(x)y-F(x)$ which satisfies our requirement. From Gauss' lemma, we may assume that F(x) and G(x) are polynomials over $\mathbb{Z}$. Let m be the maximal degree of y in $g_{i,k}(x,y)$ such that $g_{i,k}(x,\frac{F(x)}{G(x)})\cdot G(x)^m$ are over $\mathbb{Z}$ for all $1 \leq k \leq n_i$. Then consider the greatest common divisor of the contents of G(x) and $g_{i,k}(x,\frac{F(x)}{G(x)})\cdot G(x)^m$ for $1 \leq k \leq n_i$, and denote it by b. From the proof of Theorem 3.7 in \cite{Tung4}, we know that $ \forall x \exists y \varphi(x,y)$ is true in all infinite fields iff $\forall x \exists y \varphi(x,y)$ holds in all fields of characteristic 0 and $\exists y \varphi(t,y)$ holds in every rational function field $\mathbb{F}_p(t)$, where p is a prime divisor of b. Corollary 3.4 and final paragraphs of proof of Theorem 3.7 in \cite{Tung4} implies $ \forall x \exists y \varphi(x,y)$ is true in all infinite fields is decidable.
\end{proof}
From Theorem 13.4.2 in \cite{FJ}, every finitely generated transcendental extension of an arbitrary field is Hilbertian. This gives us a way to connect infinite fields and Hilbertian fields.
\begin{thm}
\label{thm18}
The $\forall^1 \exists^1$ theory of Hilbertian fields is decidable.
\end{thm}
\begin{proof}
We claim that the $\forall^1 \exists^1$ theory of Hilbertian fields and $\forall^1 \exists^1$ theory of infinite fields are identical. Then Proposition~\ref{prop13} implies the result we want. Note that all the Hilbertian fields are infinite fields. So the $\forall^1 \exists^1$ theory of infinite fields is contained in the $\forall^1 \exists^1$ theory of Hilbertian fields. Conversely, suppose that there exists an  $\forall^1 \exists^1$ sentence $\forall x \exists y \varphi(x,y)$ with $\varphi(x,y)$ quantifier-free which is true in every Hilbertian fields but false in a infinite field K. Consider the function field K(t) which is a Hilbertian field. Note that K is existentially closed in K(t) by Theorem~\ref{thm3}. Takes $m,n=1$ in Proposition~\ref{prop11} and we get $K(t) \models \exists x \forall y \neg \varphi(x,y)$ which contradicts to our assumption. This proves our claim.
\end{proof}
In the following paragraphs, we prove the decidability of the $\forall^1 \exists^1$ theory of PAC fields. The proof is similar to the proof of the decidability of the $\forall^1 \exists^1$ theory of Hilbertian fields. First, we demonstrate the case in characteristic 0.
\begin{thm}
\label{thm20}
The $\forall^1 \exists^1$ theory of PAC fields of characteristic 0 is decidable.
\end{thm}
\begin{proof}
From Corollary~\ref{cor4}, the $\forall^1 \exists^1$ theory of Hilbertian PAC fields of characteristic 0 is decidable. If we shows that the $\forall^1 \exists^1$ theory of PAC fields of characteristic 0 and the $\forall^1 \exists^1$ theory of Hilbertian PAC fields of characteristic 0 are identical, then we get the desired result.

Since every Hilbertian PAC field of characteristic 0 is a PAC field of characteristic 0, the $\forall^1 \exists^1$ theory of PAC fields of characteristic 0 is contained in the $\forall^1 \exists^1 $ theory of Hilbertian PAC fields of characteristic 0. On the other hand, let's assume that there exists an  $\forall^1 \exists^1$ sentence $\forall x \exists y \varphi(x,y)$ with $\varphi(x,y)$ quantifier-free which is true in every Hilbertian PAC field of characteristic 0 but false in a PAC field P of characteristic 0. From Theorem 13.4.6 in \cite{FJ} we can find a Hilbertian PAC field K of characteristic 0 which is regular extension of P. But Proposition 11.3.5 in \cite{FJ} shows that P is existentially closed in K. Take $m,n=1$ in Proposition~\ref{prop11} and conclude that $K \models \exists x \forall y \neg \varphi(x,y)$ which contradicts to our assumption.
\end{proof}
\begin{cor}
The $\forall^m \exists^n$ theory of PAC fields of characteristic 0 is decidable for arbitrary integers m,n.
\end{cor}
\begin{proof}
From Corollary~\ref{cor4}, the $\forall^m \exists^n$ theory of Hilbertian PAC fields of characteristic 0 is decidable for arbitrary integers m,n. Notice that Proposition~\ref{prop11} shows the preservation of $\exists^n \forall^m$ sentences for all $m,n \in \mathbb{N}$ over any $\mathcal{L}$-structures under existential closedness. Then we can modify the proof in Theorem~\ref{thm20} to show that the $\forall^m \exists^n$ theory of PAC fields of characteristic 0 and the $\forall^m \exists^n$ theory of Hilbertian PAC fields of characteristic 0 are identical. This proves the desired result.
\end{proof}
Notice that the Corollary~\ref{cor4} does not show the decidability of Hilbertian PAC field. Therefore, we need to seek other $\forall^1 \exists^1$ theory of some structures to prove the decidability of $\forall^1 \exists^1$ theory of PAC fields. From Theorem~\ref{thm17} and Proposition 11.3.5 in \cite{FJ}, we discover that the $\forall^1 \exists^1$ theory of Hilbertian fields of characteristic 0 and the $\forall^1 \exists^1$ theory of PAC fields of characteristic 0 are identical. We would generalize this result without the condition of characteristic using similar approach as before. 
\begin{thm}
\label{thm22}
The $\forall^1 \exists^1$ theory of PAC fields is decidable.
\end{thm}
\begin{proof}
We claim the the $\forall^1 \exists^1$ theory of PAC fields and the $\forall^1 \exists^1$ theory of Hilbertian fields are identical. Then Theorem~\ref{thm18} implies the theorem.

Suppose that there exists an $\forall^1 \exists^1$ sentence $\forall x \exists y \varphi(x,y)$ with $\varphi(x,y)$ quantifier-free which is true in every PAC fields but false in a Hilbertian field H. By the result in \cite{WW}, we can find a PAC field P which is totally transcendental extension of H. Using Lemma~\ref{lem3}, we have $P \models \exists x \forall y \neg \varphi(x,y)$ which contradicts to our assumption. Conversely, assume that there exists an $\forall^1 \exists^1$ sentence $\forall x \exists y \varphi(x,y)$ with $\varphi(x,y)$ quantifier-free which is true in every Hilbertian fields but false in a PAC field P. Consider the function field P(t) which is a Hilbertian field. According to Theorem~\ref{thm3}, P is existentially closed in P(t). Then take $m,n=1$ in Proposition~\ref{prop11} $ P(t) \models \exists x \forall y \neg \varphi(x,y)$ which contradicts to the assumption. Therefore, we have proved our claim.
\end{proof}
\section{Conclusion}
In this paper, we have proved the following results:
\begin{enumerate}
\item The $\exists^1 \forall^1$ theory of Hilbertian fields of characteristic 0 and perfect Hilbertian fields are decidable,
\item The $\forall^1 \exists^1$ theory of Hilbertian fields of characteristic 0 and Hilberitan fields are all decidable,
\item The $\forall^1 \exists^1$ theory of PAC fields of characteristic 0 and PAC field are decidable.
\end{enumerate}
From~\cite{Tung4}, we know that there are still no effective methods to solve the decidability of $\exists^1 \forall^1$ theory of different fields. Also, since Hilbert's tenth problem over $\mathbb{Q}$ and number fields are still open, we propose the following problems about the decidability of theories in different algebriac structures for the future developments.

\begin{enumerate}
\item $\forall^n \exists$ theory of number fields for arbitrary n,
\item $\exists^1 \forall^1$ theory of PAC fields,
\item $\exists^1 \forall^1$ theory of number fields,
\item $\exists^1 \forall^1$ theory of fields of characteristic 0,
\item $\exists^1 \forall^1$ theory of fields. 
\end{enumerate}


\begin{thebibliography}{20}
\normalsize
\addcontentsline{toc}{chapter}{Bibliography}


\bibitem{BCL}
\newblock{W.~W.~Boone, F.~B.~Cannonito and R.~C.~Lyndon,}
\newblock{\it Word Problems}, North-Holland, Amsterdam, 1973.

\bibitem{CDM}
\newblock{G.~Cherlin, L,~v.~d.~Dries, A.~MacIntyre,}
\newblock{Decidability and Undecidability theorem for PAC-fields,}
\newblock{\it Bull. Amer. Math. Soc.(N.S.)}, Vol. 4 (1981), No. 1, 101-104.

\bibitem{CK}
\newblock{C. C. Chang and H. J. Keisler,}
\newblock{\it Model theory}, Dover Publications, New York, 2012.

\bibitem{DMR}
\newblock{M.~Davis, Yu.~Matijacevi$\check{c}$ and J.~Robinson,}
\newblock{Hilbert's Tenth Problem. Diophantine equations: positive aspects of a negative solution},
\newblock{\it Proc. Symposia Pure Math.} Vol. 28, 1976, pp. 223-378.

\bibitem{ELT}
\newblock{Yu.~L.~Ershov, I.~A.~Lavrov, A.~D.~Taimanov and M.~A.~Taitslin,}
\newblock{Elementary Theories},
\newblock{\it Russian Math. Surveys}, Vol.20 (1965), No.4, 35-105.

\bibitem{KG}
\newblock{K.~G$\ddot{o}$del,}
\newblock{$\ddot{U}$ber formal unentscheidbare S$\ddot{a}$tz der Principia Mathematica und verwandter Systeme, I.}
\newblock{\it Monatsh. Math. Phys.,} No. 38 (1931), 173-198

\bibitem{FJ}
\newblock{M. D. Fried and M. Jarden,}
\newblock{\it Field Arithmetic, revised and enlarged by M. Jarden},
\newblock{Ergebnisse der Mathematik und ihrer Grenzgebiete}, vol. 11 (Second Edition), Springer-Verlag, Berlin, 2005.

\bibitem{J}
\newblock{M. Jarden,}
\newblock{\it Algebraic Patching}, 
\newblock{Springer Monograph in Mathematics}, Springer, Heidelberg, 2011.

\bibitem{JJ1}
\newblock{J.~P.~Jones,}
\newblock{Universal Diophantine equation,}
\newblock{\it The Journal of Symbolic Logic}, Vol. 47 (1982), No. 3, 549-571.

\bibitem{JJ2}
\newblock{J.~P.~Jones,}
\newblock{Classification of quantifier prefixes over Diophantine equations,}
\newblock{\it Zeitschrift fur mathematische Logik und Grundlagen der Mathematik}, Vol. 27 (1987), 403-410.

\bibitem{K}
\newblock{J.~Koenigsmann,}
\newblock{Defining $\mathbb{Z}$ in $\mathbb{Q}$,}
\newblock{\it Ann. of Math.}, Vol. 183 (2016), No. 1, 73-93.

\bibitem{CHL}
\newblock{C.~H.~Langford,}
\newblock{Theorems on deducibility,}
\newblock{\it Ann. of Math.}, (Ser. 2) Vol .28 (1927), 459-471.

\bibitem{MR}
\newblock{Yu.~Matijacevi$\check{c}$ and J.~Robinson,}
\newblock{Reduction of an arbitrary Diophantine equation to one in 13 unknowns,}
\newblock{\it Acta Arith.}, Vol. 27 (1975), 521-553. 

\bibitem{JM}
\newblock{J.~D.~Monk,}
\newblock{\it Mathematical logic},
\newblock{Graduate Text in Mathematics}, Vol. 37, Springer-Verlag, New York, 1976.

\bibitem{HO}
\newblock{H.~Ono,}
\newblock{Equational theories and universal theories of fields,}
\newblock{\it J. Math. Soc. Japan}, Vol. 35 (1983), N0.2, 289-306.

\bibitem{JP}
\newblock{J.~Park,}
\newblock{A universal first order formula defining the ring of integers in a number field,}
\newblock{\it Mathematical Research Letters}, Vol. 20 (2013), No.5, 961-980.

\bibitem{BP1}
\newblock{B.~Poonen,}
\newblock{Undecidability in number theory,}
\newblock{\it Notice Amer. Math. Soc.}, Vol 55 (2008), No. 3, 344-350.

\bibitem{BP}
\newblock{B.~Poizat,}
\newblock{\it A course in model theory: an introduction to contemporary mathematical logic}, Springer-Verlag, New York, 2000.

\bibitem{MOR}
\newblock{M.~O.~Rabin,}
\newblock{\it Decidable theories}, In
\newblock{\it Handbook of mathematical logic}, J. Barwise (ed.), North-Holland, Amsterdam, 1977, pp. 595-630.

\bibitem{RP}
\newblock{P.~Ribenboim, Remarks on existentially closed fields and diophantine equations},
\newblock{\it Rend. Sem. Mat. Univ. Padova}, vol. 71 (1984), 229-237.

\bibitem{JR}
\newblock{J.~B.~Robinson,}
\newblock{Definability and decision problems in arithmetic,}
\newblock{\it The Journal of Symbolic Logic}, Vol. 14 (1949), No. 2, 98-114.

\bibitem{JR1}
\newblock{J.~B.~Robinson,}
\newblock{The undecidability of algebraic rings and fields,}
\newblock{\it Proceedings of the American Mathematical Society}, Vol. 10 (1959), 279-284.

\bibitem{JR2}
\newblock{J.~B.~Robinson,}
\newblock{The decision problem for fields}, In
\newblock{\it The theory of Models}, C. Karp, et al., North-Holland, Amsterdam, 1965, pp. 299-311.

\bibitem{Tung1}
\newblock{S.~P.~Tung,}
\newblock{On weak number theories,}
\newblock{\it Japan. J. Math. (N.S.)}, Vol. 11 (1985), No. 2, 203-232.

\bibitem{Tung2}
\newblock{S.~P.~Tung,}
\newblock{Provability and Decidability of Arithmetical Universal-Existential Sentences},
\newblock{\it Bulletin of the London Mathematical Society}, vol. 18 (1986), No. 3, 241-247.

\bibitem{Tung3}
\newblock{S.~P.~Tung,}
\newblock{Algorithms for sentences over integral domains,}
\newblock{\it Annals of Pure and Applied Logic}, Vol. 47 (1990), No. 2, 189-197.


\bibitem{Tung4}
\newblock{S.~P.~Tung,}
\newblock{Decidable fragments of field theories,}
\newblock{\it The Journal of Symbolic Logic}, vol. 55 (1990), No. 3, 1007-1018.

\bibitem{Tung5}
\newblock{S.~P.~Tung,}
\newblock{Sentences over Integral Domains and Their Computational Complexities}
\newblock{\it Information and Computation}, Vol. 149 (1999), No. 2, 99-133.

\bibitem{Tung6}
\newblock{S.~P.~Tung, Computational complexity of sentences over fields},
\newblock{\it Information and Computation}, vol. 206 (2008), No. 7, 791-805.

\bibitem{SZW}
\newblock{Z.~W.~Sun,}
\newblock{Further Results on Hilbert's Tenth Problem,} (2017), arXiv:1704.03504.

\bibitem{WW}
\newblock{W. H. Wheeler, Model-complete theories of pseudo-algebraically closed fields},
\newblock{\it Annals of Mathematical Logic}, vol. 17 (1979), No. 3, 205-226.


\end{thebibliography}
\end{document}